\documentclass[11pt]{amsart}

\usepackage{amssymb, amsmath, amsthm}

 \newtheorem{thm}{Theorem}[section]

 \newtheorem{lem}[thm]{Lemma}
 \newtheorem{prop}[thm]{Proposition}

    \theoremstyle{definition}
 \newtheorem{defn}[thm]{Definition}
  
   \newtheorem{assn}[thm]{Assumption}

 \theoremstyle{remark}
 \newtheorem{rem}[thm]{Remark}
\numberwithin{equation}{section}

\begin{document}

\title[Automorphism groups of Weyl algebras]{Automorphism groups of Weyl algebras}

\author[No-Ho Myung]{No-Ho Myung}
\address{Department of Mathematics\\
         Chungnam  National University\\
          99 Daehak-ro,   Yuseong-gu, Daejeon 34134, Korea}
\email{nhmyung@cnu.ac.kr}             

\author[Sei-Qwon Oh]{Sei-Qwon Oh}
\address{                            
         Department of Mathematics\\
         Chungnam National University  \\
         99 Daehak-ro,   Yuseong-gu, Daejeon 34134, Korea}
\email{sqoh@cnu.ac.kr}             








\begin{abstract}
Here we construct a monomorphism from the  semigroup of all endomorphisms of the $n$-th Weyl algebra into the semigroup of all Poisson endomorphisms of the $n$-th Poisson Weyl algebra such that its restriction to the  automorphism group of the $n$-th Weyl algebra is a monomorphism into the Poisson automorphism group of the $n$-th Poisson Weyl algebra.
\end{abstract}

\maketitle


\section{Introduction}
Throughout the article, we denote by ${\bf k}$  an algebraically closed field of characteristic zero.
Suppose that $A$ is an algebra and let $\hbar\in A$ be a nonzero, nonunit, non-zero-divisor and central element such that
$A/\hbar A$ is commutative. Then $A/\hbar A$ is a nontrivial commutative algebra as well as a Poisson algebra with the Poisson bracket
\begin{equation}\label{PBRACKET2}
\{\overline{a}, \overline{b}\}=\overline{\hbar^{-1}(ab-ba)}
\end{equation}
for $\overline{a}, \overline{b}\in A/\hbar A$. Moreover, if there is an element  $0\neq\lambda\in{\bf k}$ such that $\hbar-\lambda$ is a  nonunit in $A$ then we obtain a  nontrivial algebra $A/(\hbar-\lambda) A$ with the multiplication induced by that of $A$. To make terminologies clear, we will call such an element $\hbar$ a {\it regular element} of $A$,
 the Poisson algebra $A/\hbar A$ a {\it semiclassical limit} of $A$, the algebra $A$ a {\it quantization} of $A/\hbar A$, and a nontrivial algebra $A/(\hbar-\lambda)A$ a {\it deformation} of  $A/\hbar A$.  Namely, by a regular element  $\hbar\in A$ we mean a nonzero, nonunit, non-zero-divisor and central element of $A$ such that $A/\hbar A$ is commutative.

There are many evidences that Poisson structures of Poisson algebras are analogues of algebraic structures of their quantized algebras. For example, see \cite{BrGo} and \cite{HoKa}. In \cite[Conjecture 1]{BeKo}, Kanel-Belov and Kontsevich conjectured that the automorphism group of the Weyl algebra is isomorphic to the Poisson automorphism group of the Poisson Weyl algebra. In \cite{BeGrElZh}, \cite{BeKo} and  \cite{ChOh3}, one can find positive evidences  for Kanel-Belov and Kontsevich's conjecture. A main aim of this article is to find an infinite class of deformations  of the Poisson Weyl algebra which are isomorphic to the Weyl algebra and to construct a monomorphism from the  automorphism group of the $n$-th Weyl algebra into the Poisson automorphism group of the $n$-th Poisson Weyl algebra, which is also a positive evidence  for Kanel-Belov and Kontsevich's conjecture.

Let $t$ be an indeterminate.
In section 2, we consider a ${\bf k}[t]$-algebra $A$ such that $t$ is a regular element of $A$ and that an infinite class $\widehat{A}$ of deformations of the semiclassical limit $\overline{A}:=A/tA$. Then we construct a natural map
 $\widehat{\Gamma}$ from $\widehat{A}$ onto  $\overline{A}$ and a homomorphism $\varphi$ 
from the semigroup of all endomorphisms of $\widehat{A}$  into the semigroup of all Poisson endomorphisms of $\overline{A}$ such that its restriction $\varphi|_{\widehat{A}}$ is a group homomorphism into the Poisson automorphism group of $\widehat{A}$.
 (See Proposition~\ref{RIHOMO} and Theorem~\ref{GRHOM}.)
 In  section 3, we find a ${\bf k}[t]$-algebra $A_n$ which has  a regular element $t$ such that the semiclassical limit
  $A_n/tA_n$ is Poisson isomorphic to the $n$-th Poisson Weyl algebra $B_n$ and also find an infinite class of deformations of $B_n$ that are isomorphic to the $n$-th Weyl algebra $W_n$. (See Proposition~\ref{PWA} and Proposition~\ref{PWA2}.)
 In section 4,  we construct a monomorphism from the automorphism group of $W_n$ into the Poisson automorphism group of $B_n$ by using the results in \S2. (See Theorem~\ref{WCASE}.)

 \medskip

Recall several basic terminologies.
(1) Let $S$ be a commutative ring. Given an endomorphism $\beta$ on an $S$-algebra $R$, an $S$-linear map $\nu$ is said to be a {\it left $\beta$-derivation} on $R$ if
$\nu(ab)=\beta(a)\nu(b)+\nu(a)b$ for all $a,b\in R$. For such a pair $(\beta,\nu)$, we denote by $R[z;\beta,\nu]$ the skew polynomial $S$-algebra. Note that $R[z;\beta,\nu]$ is a  free left $R$-module with basis $\{z^i\}_{i\geq0}$. Refer to \cite[\S2]{GoWa2} for details of a skew polynomial algebra.

(2) A commutative ${\bf k}$-algebra $R$ is said to be a {\it Poisson algebra} if there exists a bilinear product $\{-,-\}$ on $R$, called a {\it Poisson bracket}, such that $(R, \{-,-\})$ is a Lie algebra with $\{ab,c\}=a\{b,c\}+\{a,c\}b$ for all $a,b,c\in R$.
A derivation $\alpha$ on $R$ is said to be a {\it Poisson derivation} if $\alpha(\{a,b\})=\{\alpha(a),b\}+\{a,\alpha(b)\}$ for all $a,b\in R$. Let $\alpha$ be a Poisson derivation on $R$ and let $\delta$ be a derivation on $R$ such that
$$\delta(\{a,b\})-\{\delta(a),b\}-\{a,\delta(b)\}=\alpha(a)\delta(b)-\delta(a)\alpha(b)$$
for all $a,b\in R$. By \cite[1.1]{Oh8},  the polynomial ${\bf k}$-algebra $R[z]$ is a Poisson  algebra  with Poisson bracket $\{z,a\}=\alpha(a)z+\delta(a)$ for all $a\in R$. Such a Poisson polynomial algebra
$R[z]$ is denoted by $R[z;\alpha,\delta]_p$ in order to distinguish it from skew polynomial algebras. If $\alpha=0$ then
we write $R[z;\delta]_p$ for $R[z;0,\delta]_p$ and if $\delta=0$ then we write $R[z;\alpha]_p$ for $R[z;\alpha,0]_p$.

(3) In  Poisson algebras,  an algebra homomorphism $f$ is said to be a {\it Poisson homomorphism} if it satisfies $f(\{a,b\})=\{f(a),f(b)\}$ for all $a,b$.
For an algebra $S$, denote by $\text{End}(S)$ and $\text{Aut}(S)$ the sets of all endomorphisms and automorphisms of $S$, respectively. Likewise, for a Poisson algebra $R$, denote by $\text{P.End}(R)$ and $\text{P.Aut}(R)$ the sets of all Poisson endomorphisms and Poisson automorphisms of $R$, respectively.
Note that $\text{End}(S)$ and $\text{P.End}(R)$ are semigroups and $\text{Aut}(S)$ and $\text{P.Aut}(R)$ are groups.


\section{Automorphism groups}

Let us begin with recalling  the natural map in \cite[\S1]{Oh12}. The following assumption is a modification of   \cite[Notation 1.1]{Oh12}.

\begin{assn}\label{ASSUM}
Let $x_1,\ldots, x_n$ be indeterminates. A finite product of $x_1, x_2,\ldots, x_{n}$ is called a {\it monomial}.
For a monomial ${\bf x}=x_{i_1}\cdots x_{i_k}$, ${\bf x}$ is said to be a {\it standard} monomial if
$1\leq i_1\leq\cdots\leq i_k\leq n$.

(1) Set  ${\bf k}^*:={\bf k}\setminus\{0\}$.

(2)
Let $A$ be a ${\bf k}[t]$-algebra generated by  $x_1,\ldots, x_n$ subject to relations
$f_1,\ldots, f_r$, where ${\bf x}$ denotes a monomial in $x_1,\ldots, x_n$ and, for $i=1,\ldots,r$,
$$f_i=\sum_{\bf x}a^i_{\bf x}(t){\bf x},\ \ a^i_{\bf x}(t)\in{\bf k}[t].$$

(3) Assume that $t\in{\bf k}[t]$ is a regular element in $A$ and thus the semiclassical limit
$$\overline{A}:= A/tA$$ is a Poisson algebra with Poisson bracket (\ref{PBRACKET2}).

(4) Assume that, for each $q\in{\bf k}^*$,  deformation $A_q:=A/(t-q)A$ is nonzero.
 Hence $\overline{A}$ and $A_q$ are ${\bf k}$-algebras generated by $x_1,\ldots,x_n$
subject to the relations
$$f_1|_{t=0},\ldots, f_r|_{t=0}$$ and $$f_1|_{t=q},\ldots, f_r|_{t=q},$$ respectively.

(5) Assume that $A$ has a ${\bf k}[t]$-basis $\{\xi_i|i\in I\}$  such that, for each $q\in {\bf k}^*$,
 $$\{\overline{\xi_i}:=\xi_i+tA|i\in I\},\ \ \ \{\xi_i+(t-q)A|i\in I\}$$ are ${\bf k}$-bases of $\overline{A}$ and $A_q$, respectively.
Hence every element $f(t)\in A$ is expressed uniquely by
$$f(t)=\sum_ia_i(t)\xi_i,\ \ a_i(t)\in{\bf k}[t]$$
and, for each $q\in {\bf k}^*$,
$$\begin{array}{c} f(t)+(t-q)A=\sum_ia_i(q)(\xi_i+(t-q)A)\in A_q,\ \ \
\overline{f(t)}=\sum_ia_i(0)\overline{\xi_i}\in \overline{A}.
\end{array}$$

We will still write $\xi_i$ for $\overline{\xi_i}\in \overline{A}$ and $\xi_i+(t-q)A\in A_q$ if no confusion arises.
Thus, for $f(t)\in A$, the expressions
   $$\overline{f(t)}=f(t)|_{t=0}\in \overline{A},\ \ \ f(t)+(t-q)A=f(t)|_{t=q}\in A_q$$
make sense. For $f(t)\in A$, we often write $f(0)$ and $f(q)$ for $f(t)|_{t=0}$ and $f(t)|_{t=q}$, respectively.
\end{assn}

\begin{rem}\label{GEN}
Since ${\bf k}^*$ is an infinite set, we obtain  the infinite class
$${\bf D}:=\{A_q| q\in{\bf k}^*\}$$ consisting of deformations $A_q$ which are nontrivial ${\bf k}$-algebras by Assumption~\ref{ASSUM}(4).

Let $\widehat{q}$ be the parameter taking values in ${\bf k}^*$. That is, $\widehat{q}$ is the function from ${\bf k}^*$ into
${\bf k}$ defined by
$$\widehat{q}(q)=q$$
for each $q\in{\bf k}^*$. Denote by
  $$\text{$A_{\widehat{q}}$ \ the ${\bf k}$-algebra obtained from $A_q$ by replacing $q\in A_q$ by $\widehat{q}$.}$$
Thus $A_{\widehat{q}}$ means the class ${\bf D}$. Note that
$$\widehat{q}^{-1}\in A_{\widehat{q}}$$
since  $q^{-1}\in A_q$ for each $q\in{\bf k}^*$. For a nonzero polynomial $f(t)\in{\bf k}[t]$, $f(\widehat{q})$ is invertible in
$A_{\widehat{q}}$ if and only if $f(q)\neq0$ for each $q\in{\bf k}^*$. Hence, since ${\bf k}$ is an algebraically closed field, $f(\widehat{q})$ is invertible in
$A_{\widehat{q}}$ if and only if $f(\widehat{q})=a\widehat{q}^k$ for some $0\neq a\in{\bf k}$ and  nonnegative integer $k$.
\end{rem}

\bigskip

Let $\gamma$  denote the ${\bf k}$-algebra homomorphism
$$\gamma:A\to \prod_{q\in{\bf k}^*}A_q,\ \ f(t)\mapsto (f(q))_{q\in{\bf k}^*}. $$
Then $\gamma$ is injective by \cite[Lemma 1.2]{Oh12} and thus there exists  the composition
$$\Gamma=\overline{\gamma}\circ\gamma^{-1}:\gamma(A)\overset{\gamma^{-1}}{\longrightarrow} A\overset{\overline{\gamma}}{\longrightarrow} \overline{A}=A/tA,$$
where $\overline{\gamma}:A\to\overline{A}$ is the canonical projection.
Denote by \ $\widehat{}$ \ the ${\bf k}$-algebra homomorphism
$$\widehat{}\ :A_{\widehat{q}}\longrightarrow \prod_{q\in{\bf k}^*}A_q,\ \ f\mapsto (f|_{\widehat{q}=q})_{q\in{\bf k}^*}.$$
Let $\widehat{A}$ be the inverse image of $\gamma(A)$ by \ $\widehat{}$\ .  Namely, $\widehat{A}=\widehat{}\ ^{-1}(\gamma(A))$.

\begin{lem}\label{SHAT}
(1) $\widehat{A}=\{f(\widehat{q})\in A_{\widehat{q}}\ |\ f(t)\in A\}.$

(2)  $\{\widehat{q}^j\xi_i\ | j\geq0, i\in I\}$ is a ${\bf k}$-basis of $\widehat{A}$.
\end{lem}

\begin{proof}
(1) For every $f(t)=\sum_{i\in I} a_i(t)\xi_i\in A$ ($a_i(t)\in{\bf k}[t]$), we have that
$\gamma(f(t))=\widehat{}\ (f(\widehat{q}))$ and thus the result follows.

(2) Since $A$ is a ${\bf k}[t]$-algebra,  $\widehat{A}$ is  spanned ${\bf k}$-linearly by elements $\widehat{q}^j\xi_i$ for  $j\geq0$ and $i\in I$ by (1).
Suppose that $\sum_{j,i}c_{ji}\widehat{q}^j\xi_i=0$ for some $c_{ji}\in{\bf k}$. Setting $a_i(t)=\sum_j c_{ji}t^j\in{\bf k}[t]$, we have that
$$0=\sum_{j,i}c_{ji}\widehat{q}^j\xi_i=\sum_i(\sum_j c_{ji}\widehat{q}^j)\xi_i=\sum_ia_i(\widehat{q})\xi_i.$$
Since  $\widehat{q}$ can take infinitely many elements in ${\bf k}^*$,  $a_i(t)$ has infinitely many zeros   and thus
$a_i(t)=0$. It follows that $c_{ji}=0$ and thus  $\{\widehat{q}^j\xi_i\ | j\geq0, i\in I\}$ is a ${\bf k}$-basis of $\widehat{A}$.
\end{proof}

Note by Lemma~\ref{SHAT}(1) that  $\widehat{A}$ is a ${\bf k}$-subalgebra of $A_{\widehat{q}}$ such that $\widehat{q}^{-1}\notin \widehat{A}$ since $t^{-1}\notin A$.
Denote by $\widehat{\Gamma}$  the composition
\begin{equation}\label{NATM}
\widehat{\Gamma}=\Gamma\circ \widehat{}\ :\widehat{A}\overset{\widehat{}}{\longrightarrow}\gamma(A)\overset{\Gamma}{\longrightarrow} \overline{A},\ \ f(\widehat{q})\mapsto f(0).
\end{equation}

\begin{lem}\label{SURJ}
For   $f(t)\in A$,
  $\widehat{\Gamma}(f(\widehat{q}))={f(0)}=\overline{f(t)}$.
\end{lem}

\begin{proof}
Since $f(\widehat{q})\in \widehat{A}$ by Lemma~\ref{SHAT}, it  follows by Assumption~\ref{ASSUM}(5).
\end{proof}

\begin{prop}\label{RIHOMO}
The map (\ref{NATM}) is a ${\bf k}$-algebra  epimorphism such that
$$\widehat{\Gamma}(x_i)=x_i,\ \ \widehat{\Gamma}(\widehat{q})=0,\ \
 \widehat{\Gamma}(c)=c$$ for  $i=1,\ldots, n$ and each $c\in{\bf k}$.
\end{prop}

\begin{proof}
Since \ $\widehat{}$ \   and $\Gamma$ are  ${\bf k}$-algebra homomorphisms,
$\widehat{\Gamma}$ is a ${\bf k}$-algebra homomorphism. It is clear that $\widehat{\Gamma}(\widehat{q})=0$,
 $\widehat{\Gamma}(c)=c$ for each $c\in{\bf k}$ and $\widehat{\Gamma}(x_i)=x_i$ for $i=1,\ldots, n$.
  The map $\widehat{\Gamma}$ is surjective  by Lemma~\ref{SURJ}.
\end{proof}

We will construct a homomorphism $\varphi$ from $\text{End}(\widehat{A})$  into $\text{P.End}(\overline{A})$.
For $\sigma\in\text{End}(\widehat{A})$, define a map
\begin{equation}\label{GROUPHOM}
\varphi(\sigma):\overline{A}\longrightarrow \overline{A},\ \ \varphi(\sigma)(z)=\widehat{\Gamma}\sigma(z'),
\end{equation}
 where $\widehat{\Gamma}(z')=z$.
Namely, $\varphi(\sigma)(z)=\widehat{\Gamma}\sigma{\widehat{\Gamma}}^{-1}(z)$.

\begin{lem}\label{GROUP}
For $\sigma\in\text{End}(\widehat{A})$, $\varphi(\sigma):\overline{A}\to \overline{A}$ is a Poisson homomorphism such that
$$\varphi(\sigma)(x_i)=\sigma(x_i)|_{\widehat{q}=0}$$ for all $i=1,\ldots,n$.
In particular, $\varphi(\text{id}_{\widehat{A}})=\text{id}_{\overline{A}}$.

$$\begin{picture}(300,100)(-50,-40) 
                                                            \put(80,50){$\sigma$}
                             \put(5,40){$\widehat{A}$}      \put(20,45){\vector(1,0){135}}   \put(160,40){$\widehat{A}$}
\put(0,0){$\widehat{\Gamma}$}\put(10,30){\vector(0,-1){50}}                                  \put(165,30){\vector(0,-1){50}}\put(170,0){$\widehat{\Gamma}$}
                             \put(5,-40){$\overline{A}$}    \put(20,-35){\vector(1,0){135}}  \put(160,-40){$\overline{A}$}
                                                            \put(80,-45){$\varphi(\sigma)$}

                             \put(27,20){$x_i$}             \put(45,22){\vector(1,0){65}}    \put(115,20){$\sigma(x_i)$}
                             \put(15,35){\line(1,-1){10}}                                    \put(160,35){\line(-2,-1){15}}
                             \put(32,-10){\vector(0,1){20}}                                  \put(130,10){\vector(0,-1){20}}
                             \put(15,-30){\line(1,1){10}}                                    \put(160,-30){\line(-2,1){10}}
                             \put(27,-20){$x_i$}            \put(45,-17){\vector(1,0){65}}   \put(115,-20){$\sigma(x_i)|_{\hat{q}=0}$}
\end{picture}$$
\bigskip
\end{lem}

\begin{proof}
For any $z\in \overline{A}$, suppose that $f(0)=g(0)=z$ for some $f(t),g(t)\in A$.
Then  $f(\widehat{q}), g(\widehat{q})\in \widehat{A}$ and
$$\widehat{\Gamma}(f(\widehat{q}))=f(0)=z=g(0)=\widehat{\Gamma}( g(\widehat{q}))$$
 by Lemma~\ref{SURJ}. Note that  $f(\widehat{q})-g(\widehat{q})$ is expressed by $$f(\widehat{q})-g(\widehat{q})=\sum_{i}a_i(\widehat{q})\xi_i$$ for some  $a_i(t)\in{\bf k}[t]$ by Lemma~\ref{SHAT}. Since
$f(0)-g(0)=0$,  $a_i(0)=0$ for all $i\in I$. Hence
$$(\widehat{\Gamma}\sigma)(f(\widehat{q})-g(\widehat{q}))=(\sigma(f(\widehat{q})-g(\widehat{q})))|_{\widehat{q}=0}=\sum_{i\in I}a_i(0)(\sigma(\xi_i))|_{\widehat{q}=0}=0$$
since $a_i(0)=0$  for all $i$. Hence $(\widehat{\Gamma}\sigma)(f(\widehat{q}))=(\widehat{\Gamma}\sigma)(g(\widehat{q}))$
and thus $\varphi(\sigma)$ is well defined.

By Proposition~\ref{RIHOMO}, $\varphi(\sigma)$ is a ${\bf k}$-algebra homomorphism  and
$$\begin{aligned}
\varphi(\sigma)(x_i)=\widehat{\Gamma}(\sigma(x_i))=\sigma(x_i)|_{\widehat{q}=0}
\end{aligned}$$
for all $i=1,\ldots,n$.
Hence $\varphi(\text{id}_{\widehat{A}})=\text{id}_{\overline{A}}$.

For $a,b\in \overline{A}$, let $f(t),g(t)\in A$ such that $\overline{f(t)}=a, \overline{g(t)}=b$.
Then  $f(\widehat{q}), g(\widehat{q})\in \widehat{A}$ and
$$\widehat{\Gamma}(f(\widehat{q}))=a=\overline{f(t)},\ \ \widehat{\Gamma}(g(\widehat{q}))=b=\overline{g(t)}$$
by Lemma~\ref{SURJ}. Thus
$$\overline{t^{-1}(f(t)g(t)-g(t)f(t))}
={\widehat{\Gamma}}(\widehat{q}^{-1}(f(\widehat{q})g(\widehat{q})-g(\widehat{q})f(\widehat{q}))).$$
Therefore
$$\begin{aligned}
\varphi(\sigma)(\{a,b\})&=\widehat{\Gamma}\sigma(\widehat{q}^{-1}(f(\widehat{q})g(\widehat{q})-g(\widehat{q})f(\widehat{q})))\\
&=\left(\widehat{q}^{-1}(\sigma(f(\widehat{q}))\sigma(g(\widehat{q}))-\sigma(g(\widehat{q}))\sigma(f(\widehat{q})))\right)_{\widehat{q}=0}\\
&=\overline{t^{-1}(\sigma(f(\widehat{q}))|_{\widehat{q}=t}\sigma(g(\widehat{q}))|_{\widehat{q}=t}-\sigma(g(\widehat{q}))|_{\widehat{q}=t}\sigma(f(\widehat{q}))|_{\widehat{q}=t})}\\
&=\{\varphi(\sigma)(a), \varphi(\sigma)(b)\}
\end{aligned}$$
by Lemma~\ref{SURJ}
and thus $\varphi(\sigma)$ is a Poisson homomorphism.
\end{proof}

\begin{thm}\label{GRHOM}
The map
$$\varphi: \text{End}(\widehat{A})\to\text{P.End}(\overline{A}),\ \ \sigma\mapsto\varphi(\sigma)$$
is a semigroup homomorphism such that its restriction
$$\varphi|_{\text{Aut}(\widehat{A})}: \text{Aut}(\widehat{A})\to\text{P.Aut}(\overline{A})$$
is a group homomorphism.
\end{thm}

\begin{proof}
If $\sigma, \tau\in\text{End}(\widehat{A})$ then
$$\varphi(\sigma\tau)=\widehat{\Gamma}(\sigma\tau){\widehat{\Gamma}}^{-1}
=(\widehat{\Gamma}\sigma{\widehat{\Gamma}}^{-1})({\widehat{\Gamma}}\tau{\widehat{\Gamma}}^{-1})
=\varphi(\sigma)\varphi(\tau)$$
and thus $\varphi$ is a semigroup homomorphism by Lemma~\ref{GROUP}.

Let $\sigma\in\text{Aut}(\widehat{A})$. Then $\sigma^{-1}\in\text{Aut}(\widehat{A})$ and thus
$$\begin{aligned}
\varphi(\sigma)\varphi(\sigma^{-1})&=\varphi(\sigma\sigma^{-1})=\varphi(\text{id}_{\widehat{A}})
=\text{id}_{\overline{A}},\\
\varphi(\sigma^{-1})\varphi(\sigma)&=\varphi(\sigma^{-1}\sigma)=\varphi(\text{id}_{\widehat{A}})
=\text{id}_{\overline{A}}.
\end{aligned}$$
Hence $\varphi(\sigma)$ is a Poisson automorphism by Lemma~\ref{GROUP}. It follows that
$\varphi|_{\text{Aut}(\widehat{A})}$ is a group homomorphism from $\text{Aut}(\widehat{A})$ into
$\text{P.Aut}(\overline{A})$.
\end{proof}


\section{Quantization of Poisson Weyl algebras}

Here we obtain a quantization  of Poisson Weyl algebra.

\begin{defn}
The  {\it $n$-th Poisson Weyl algebra} is
 the ${\bf k}$-algebra $B_n={\bf k}[x_1, \ldots, x_{2n-1},x_{2n}]$ equipped with the Poisson bracket
$$\{f,g\}=\sum_{i=1}^n\left(-\frac{\partial f}{\partial x_{2i-1}}\frac{\partial g}{\partial x_{2i}}+\frac{\partial g}{\partial x_{2i-1}}\frac{\partial f}{\partial x_{2i}}\right)$$
for all $f,g\in B_n$.
(See \cite[1.1.A]{ChPr} and \cite[1.3]{MeHaO}.)
\end{defn}

Note that $B_n$ is Poisson simple, namely there is no nontrivial Poisson ideal in $B_n$,
and that   $B_n$ is a Poisson algebra
with the Poisson bracket: for $j>i$,
$$\{x_j,x_i\}=\left\{\begin{aligned} &1,&&\text{if }j=2\ell,i=2\ell-1,\\ &0,&&\text{otherwise}.\end{aligned}\right.$$
Thus $B_n$ is an iterated Poisson polynomial algebra
$$B_n={\bf k}[x_1][x_2;\delta_2]_p[x_3]_p[x_4;\delta_4]_p\ldots[x_{2n-1}]_p[x_{2n};\delta_{2n}]_p,$$ where
$$\delta_{2j}(x_i)=\left\{\begin{aligned} &1,& &\text{if }i=2j-1,\\ &0,&&\text{if }i\neq 2j-1.\end{aligned}\right.$$

\begin{prop}\label{PWA}
Let  $A_n$ be an ${\bf k}[t]$-algebra generated by $x_1, \ldots, x_{2n-1},x_{2n}$ subject to the relations, for $j>i$
\begin{equation}\label{YR1}
x_jx_i=\left\{\begin{aligned}&x_ix_j+t,&&\text{if }j=2\ell,i=2\ell-1\\&x_ix_j,&&\text{otherwise}.\end{aligned}\right.
\end{equation}
Then $t$ is a regular element in $A_n$ and the semiclassical limit $\overline{A_n}:=A_n/tA_n$ is Poisson isomorphic to the $n$-th Poisson Weyl algebra $B_n$. (We will identify $\overline{A_n}$ to $B_n$.)
In particular, $A_n$ is an iterated skew polynomial ${\bf k}[t]$-algebra
$$A_n={\bf k}[t][x_1][x_2; \nu_2]\ldots[x_{2n-1}][x_{2n};\nu_{2n}],$$
where ${\bf k}[t]$-linear map $\nu_{2j}$ is a derivation defined by
$$ \nu_{2j}(x_i)=\left\{\begin{aligned}&t,&&\text{if } i=2j-1,\\ &0,&&\text{otherwise.}\end{aligned}\right.$$
Namely, $\nu_{2j}=t\frac{\partial}{\partial x_{2j-1}}$.
\end{prop}

\begin{proof}
By (\ref{YR1}), it is observed that
$A_n$ is   an iterated skew polynomial ${\bf k}[t]$-algebra
$$A_n={\bf k}[t][x_1][x_2; \nu_2]\ldots[x_{2n-1}][x_{2n};\nu_{2n}],$$
where
$$ \nu_{2j}(x_i)=\left\{\begin{aligned}&t,&&\text{if } i=2j-1,\\ &0,&&\text{otherwise.}\end{aligned}\right.$$
Hence $A_n$ is a domain and thus  the central element $t\in A_n$ is a nonzero, nonunit and non-zero-divisor such that
$A_n/tA_n$ is commutative by (\ref{YR1}).
It follows that   $t$ is a regular element of $A_n$ and   the semiclassical limit $\overline{A_n}=A_n/tA_n$ is a Poisson algebra with Poisson bracket
$$\{\overline{x_j},\overline{x_i}\}=\overline{t^{-1}(x_jx_i-x_ix_j)}=
 \left\{\begin{aligned} &1,&&\text{if }j=2\ell,i=2\ell-1,\\ &0,&&\text{otherwise}.\end{aligned}\right.$$
 Thus $\overline{A_n}$ is Poisson isomorphic to $B_n$.
\end{proof}

\begin{defn}
The  {\it $n$-th Weyl algebra} $W_n$ is a  ${\bf k}$-algebra generated by $x_1, \ldots, x_{2n-1},x_{2n}$ subject to the relations, for $j>i$
$$x_jx_i-x_ix_j=\left\{\begin{aligned}&1 , &&\text{if }j=2\ell, i=2\ell-1,\\ &0,&&\text{otherwise.}\end{aligned}\right.$$
\end{defn}

\begin{prop}\label{PWA2}
Retain the notations in Proposition~\ref{PWA}.

(1) For each $q\in {\bf k}^*$,  let $A_q:=A_n/(t-q)A_n$. Then $A_q$ is a ${\bf k}$-algebra
generated by
$x_1,\ldots, x_{2n-1},x_{2n}$ subject to the relations
$$x_jx_i-x_ix_j=\left\{\begin{aligned}&q , &&\text{if }j=2\ell, i=2\ell-1,\\ &0,&&\text{otherwise,}\end{aligned}\right.$$
which is an iterated skew polynomial ${\bf k}$-algebra
$$A_q={\bf k}[x_1][x_2; \nu_2']\ldots[x_{2n-1}][x_{2n};\nu_{2n}'],$$
where ${\bf k}$-linear map  $\nu'_{2j}$ is a derivation defined by
$$ \nu'_{2j}(x_i)=\left\{\begin{aligned}&q,&&\text{if } i=2j-1,\\ &0,&&\text{otherwise.}\end{aligned}\right.$$
Namely, $\nu'_{2j}=q\frac{\partial}{\partial x_{2j-1}}$.

(2) For each $q\in {\bf k}^*$, $A_q\cong W_n$. In particular, if $q=1$ then $A_q=W_n$.

(3) The set of all standard monomials in $x_1, x_2,\ldots, x_{2n}$ including the unity $1$ is a ${\bf k}[t]$-basis of
$A_n$ and a ${\bf k}$-basis of $A_q, W_n, B_n$ for each $q\in {\bf k}^*$.
\end{prop}

\begin{proof}
(1) It follows immediately by (\ref{YR1}) and Proposition~\ref{PWA} by replacing $t$ and ${\bf k}[t]$ by $q$ and ${\bf k}$, respectively.

(2) The map from $A_q$ into $W_n$ defined by
$$
x_i\mapsto \left\{\begin{aligned}&x_i , &&\text{if } i=2\ell-1,\\ &qx_i,&&\text{if }i=2\ell\end{aligned}\right.
$$
is a ${\bf k}$-algebra isomorphism since $q$ is nonzero. Hence $A_q\cong W_n$.

(3) Since $A_n$, $A_q$ and $W_n$ are iterated skew polynomial algebras, the set of all standard monomials including the unity $1$ is their bases by \cite[\S2]{GoWa2}. Since $B_n$ is the commutative polynomial ${\bf k}$-algebra ${\bf k}[x_1, x_2,\ldots, x_{2n}]$,
the set of all standard monomials including the unity $1$ is a ${\bf k}$-basis of $B_n$ clearly.
\end{proof}


\section{Application to Weyl algebra}

Henceforth, retain the notations of \S3.
By  Proposition~\ref{PWA2}, the set $\{\xi_i | i\in I\}$ of all standard monomials including the unity $1$ is a ${\bf k}[t]$-basis of $A_n$ and ${\bf k}$-bases of $A_q$ and $B_n$. Hence $A_n, B_n$ and $A_q$ $(q\in{\bf k}^*)$ satisfy (1)-(5) of Assumption~\ref{ASSUM}.
Since $q$ works   as  a parameter taking values in ${\bf k}^*$, we get a ${\bf k}$-algebra $A_{\widehat{q}}$ from $A_q$ by replacing $q$  in $A_q$ by $\widehat{q}$. Denote by
$\widehat{A_n}$ the ${\bf k}$-subalgebra \ $\widehat{}\ ^{-1}(\gamma(A_n))$ of $A_{\widehat{q}}$ as in \S2.

By the proof of Proposition~\ref{PWA2}(2), the map
defined by
\begin{equation}\label{OH1}
\psi:A_{\widehat{q}}\to W_n,\ \ \ \psi(x_i)= \left\{\begin{aligned}&x_i , &&\text{if } i=2\ell-1,\\ &\widehat{q}x_i,&&\text{if }i=2\ell\end{aligned}\right.
\end{equation}
is an isomorphism since $\widehat{q}$ is invertible.

\begin{lem}\label{OHH}
Let $\sigma\in\text{End}(W_n)$. Then, for any $f\in \widehat{A_n}$, there exists a
positive integer $k(f)$ such that $(\widehat{q}+1)^{k(f)}(\psi^{-1}\sigma\psi)(f)\in (\widehat{q}+1)\widehat{A_n}$.
In particular, if $\sigma(x_i)=\sum_{j\in I}a_j\xi_j$ for some $a_j\in{\bf k}$ then
there exits a positive integer $k(x_i)$ such that
\begin{equation}\label{OOH}
(\widehat{q}+1)^{k(x_i)}(\psi^{-1}\sigma\psi)(x_i)=\sum_{j\in I} a_j(\widehat{q}+1)^{\ell_j}\xi_j,
\end{equation}
where all $\ell_j$ are positive integers.
\end{lem}

\begin{proof}
Note that $\psi^{-1}\sigma\psi$ is an endomorphism of $A_{\widehat{q}}$ by (\ref{OH1}). Hence it is enough to prove (\ref{OOH}).
By (\ref{OH1}), we have that
$$\begin{aligned}
(\psi^{-1}\sigma\psi)(x_i)&=\left\{\begin{aligned}&\psi^{-1}\sigma(x_i)=\sum_{j\in I}a_j\psi^{-1}(\xi_j) , &&\text{if } i=2\ell-1,\\ &\psi^{-1}\sigma(\widehat{q}x_i)=\sum_{j\in I}a_j\widehat{q}\psi^{-1}(\xi_j),&&\text{if }i=2\ell\end{aligned}\right.\\
&=\sum_{j\in I}a_j \widehat{q}^{m_j}\xi_j\end{aligned}
$$
 for some integers $m_j$ since
$\psi^{-1}(\xi_j)=\widehat{q}^m\xi_j$ for some non-positive integer $m$ by (\ref{OH1}). Multiplying $\widehat{q}^k$ in the above equation
for a sufficiently large positive integer $k$,   we get
\begin{equation}\label{OH2}
\widehat{q}^k(\psi^{-1}\sigma\psi)(x_i)=\sum_{j\in I} a_j\widehat{q}^{\ell_j}\xi_j,
\end{equation}
where all $\ell_j$ are positive integers.

If $\widehat{q}=-1$ then (\ref{OOH}) holds since the both sides  are zero.
 If $\widehat{q}\neq-1$ then $\widehat{q}+1$ is nonzero. Since $\widehat{q}$ is the parameter taking values in ${\bf k}^*$, we may replace $\widehat{q}$ in (\ref{OH2}) by $\widehat{q}+1$, and thus (\ref{OOH}) holds.
\end{proof}

\begin{thm}\label{WCASE}
There exists a semigroup monomorphism $\Psi:\text{End}(W_n)\to \text{P.End}(B_n)$ such that its restriction
$\Psi |_{\text{Aut}(W_n)}:\text{Aut}(W_n)\to \text{P.Aut}(B_n)$ is a group monomorphism.
\end{thm}

\begin{proof}
Let $\sigma\in \text{End}(W_n)$ and $g\in B_n$. By Proposition~\ref{RIHOMO}, there exists an element
$f\in \widehat{A_n}$ such that $\widehat{\Gamma}(f)=g$ and, by Lemma~\ref{OHH},
there exists a positive integer $k(f)$ such that
$$(\psi^{-1}\sigma\psi)((\widehat{q}+1)^{k(f)}f)=(\widehat{q}+1)^{k(f)}(\psi^{-1}\sigma\psi)(f)\in (\widehat{q}+1)\widehat{A_n}\subseteq\widehat{A_n}.$$
Note that  $\widehat{\Gamma}((\widehat{q}+1)^{k(f)}f)=(\widehat{\Gamma}(\widehat{q}+1))^{k(f)}\widehat{\Gamma}(f)=g$.
Set
$$\Psi(\sigma)(g):=\widehat{\Gamma}(\psi^{-1}\sigma\psi)((\widehat{q}+1)^{k(f)}f)
=\widehat{\Gamma}(\psi^{-1}\sigma\psi)\widehat{\Gamma}^{-1}(g).$$
$$\begin{picture}(400,200)(-100,-100) 
                                                               \put(100,88){$\sigma$}
                                \put(-10,80){$W_n$}            \put(15,83){\vector(1,0){185}}    \put(210,80){$W_n$}
\put(-15,45){$\psi$}            \put(-5,12){\vector(0,1){60}}                                    \put(215,75){\vector(0,-1){60}}\put(220,45){$\psi^{-1}$}
                                \put(-10,2){$A_{\widehat{q}}$}                                   \put(210,2){$A_{\widehat{q}}$}
                                \put(10,2){$\supseteq$}        \put(85,10){$\psi^{-1}\sigma\psi$}\put(197,2){$\subseteq$}
                                \put(25,0){$\widehat{A_n}$}    \put(40,5){\vector(1,0){135}}     \put(180,0){$\widehat{A_n}$}
\put(20,-35){$\widehat{\Gamma}$}\put(30,-10){\vector(0,-1){55}}                            \put(185,-10){\vector(0,-1){55}}\put(190,-35){$\widehat{\Gamma}$}
                                \put(60,-70){$\Psi(\sigma)=\widehat{\Gamma}(\psi^{-1}\sigma\psi)\widehat{\Gamma}^{-1}$}
                                \put(25,-80){$B_n$}            \put(40,-75){\vector(1,0){135}}   \put(180,-80){$B_n$}
\end{picture}$$
Note that $\Psi(\sigma)$ is determined independently to the choice of positive integer $k(f)$ by the first paragraph in the proof of Lemma~\ref{GROUP} and thus we may assume that $\psi^{-1}\sigma\psi\in \text{End}(\widehat{A_n})$.
It follows that $\Psi(\sigma)=\varphi(\psi^{-1}\sigma\psi)$ by (\ref{GROUPHOM})
and thus
$\Psi$ is a semigroup
homomorphism from $\text{End}(W_n)$ into $\text{P.End}(B_n)$ such that its restriction
$\Psi|_{\text{Aut}(W_n)}$ is a group
homomorphism  into $\text{P.Aut}(B_n)$  by Theorem~\ref{GRHOM} and (\ref{OH1}).

Suppose that $\sigma,\mu\in\text{End}(W_n)$ such that $\sigma\neq\mu$. Then,  for some $i$, $\sigma(x_i)\neq\mu(x_i)$ in $W_n$.
Hence $\Psi(\sigma)(x_i)=\sigma(x_i)\neq\mu(x_i)=\Psi(\mu)(x_i)$ in $B_n$ by (\ref{OOH}) and thus
$\Psi(\sigma)\neq\Psi(\mu)$. It follows that $\Psi$ is a monomorphism.
\end{proof}

\begin{rem}
(1) Kanel-Belov and Kontsevich conjectured in \cite[Conjecture 1]{BeKo} that
the automorphism group of the $n$-th Weyl algebra $W_n$  is isomorphic to the Poisson automorphism group of the $n$-th Poisson Weyl algebra $B_n$. We expect that $\Psi |_{\text{Aut}(W_n)}$ in Theorem~\ref{WCASE} is an isomorphism and thus that the  Kanel-Belov and Kontsevich's conjecture is true.

(2) Recall Dixmier's conjecture: Every endomorphism of the $n$-th Weyl algebra is an automorphism. Here is the Poisson analogue of Dixmier's conjecture:  Every Poisson endomorphism of the $n$-th Poisson Weyl algebra is a Poisson automorphism.
\end{rem}

\noindent
{\bf Acknowledgments} The second author is supported by National  Research Foundation of Korea,  NRF-2017R1A2B4008388,
and thanks the Korea Institute for Advanced Study for the warm hospitality during the preparation of this article.

\bibliographystyle{amsplain}

\providecommand{\bysame}{\leavevmode\hbox to3em{\hrulefill}\thinspace}
\providecommand{\MR}{\relax\ifhmode\unskip\space\fi MR }
\providecommand{\MRhref}[2]{%
  \href{http://www.ams.org/mathscinet-getitem?mr=#1}{#2}
}
\providecommand{\href}[2]{#2}

\end{document}